\documentclass[10pt]{amsart}
\usepackage{amssymb}
\usepackage{bm}
\usepackage{graphicx}
\usepackage[centertags]{amsmath}
\usepackage{amsfonts}
\usepackage{amsthm}
\usepackage{graphicx}
\newtheorem{thm}{Theorem}
\newtheorem{cor}[thm]{Corollary}
\newtheorem{lem}[thm]{Lemma}

\newtheorem{prop}[thm]{Proposition}

\newtheorem{defn}[thm]{Definition}
\theoremstyle{definition}

\newtheorem{notation}{Notation}

\newcommand{\nn}{\mathbb{N}}
\newcommand{\ee}{\varepsilon}
\newcommand{\meg}{\geqslant}
\newcommand{\mik}{\leqslant}

\newcommand{\con}{\smallfrown}

\makeindex

\begin{document}

\title{Subsets of Products of Finite Sets of Positive Upper Density}
\author{Stevo Todorcevic}
\author{Konstantinos Tyros}

\address{Department of Mathematics, University of Toronto, Toronto, Canada, M5S 2E4}
\address{and}
\address{Institute Math\'ematiques de Jussieu, CNRS UMR 7586, 2 Place Jussieu- Case
7012, 72251 Paris, Cedex 05, France. }
\email{stevo@math.toronto.edu}
\address{and}
\address{Department of Mathematics, University of Toronto, Toronto, Canada, M5S 2E4 }
\email{ktyros@math.toronto.edu}

\thanks{The first named author is supported by grants from NSERC and CNRS}
\begin{abstract}
In this note we prove that for every sequence $(m_q)_{q}$ of positive integers and for every real $0<\delta\leqslant1$ there is a sequence
$(n_q)_{q}$ of positive integers such that for every sequence $(H_q)_{q}$ of finite sets such that $|H_q|=n_q$ for every $q\in\mathbb{N}$ and for every $D\subseteq \bigcup_k\prod_{q=0}^{k-1}H_q$ with the property that
$$\limsup_k \frac{|D\cap \prod_{q=0}^{k-1} H_q|}{|\prod_{q=0}^{k-1}H_q|}\geqslant\delta$$
there is a sequence $(J_q)_{q}$, where $J_q\subseteq H_q$ and $|J_q|=m_q$ for all $q$, such that $\prod_{q=0}^{k-1}J_q\subseteq D$ for infinitely many $k.$ This gives us a density version of a well-known Ramsey-theoretic result. We also give some estimates on the sequence $(n_q)_{q}$
in terms of the sequence of $(m_q)_{q}$.
\end{abstract}
\maketitle

\section{Introduction}

It is well known that many Ramsey-theoretic results admit density versions. For example, Szemeredi's theorem \cite{Sz} which is just a density version of a much older and as famous Ramsey-theoretic result of van der Wearden \cite{vW}.  Density versions of Ramsey-theoretic results tend to be considerably harder to prove. A typical such example is the density version of the Hales-Jewett theorem \cite{HJ} due to Furstenberg and Katznelson \cite{FK} (see also \cite{Po}). All these examples belong to finite Ramsey theory in the sense of \cite{GRS}. However, we have  recently  seen that infinite Ramsey-theoretic results sometimes do admit density version in spite of the well-known fact that the original infinite  Ramsey theorem itself (\cite{Ra}) does not have a density version (see, for example, \cite{EG}). Recently, solving an old problem of Laver,  Dodos, Kanellopoulos and Karagiannis \cite{DKK}, have proved a density version of the famous Halpern-L\"auchli theorem \cite{HL}.  This result has in turn triggered  investigations of density versions of the finite form of the Halpern-L\"auchli theorem (see \cite{DKT}). In this note we prove a density version of another infinite Ramsey-theoretic result,  a Ramsey type result about products of finite sets (see \cite{DLT}, \cite{To}, Chapter 3)  stating that for every infinite sequence $(m_q)_{q}$ of positive integers there is a sequence $(n_q)_{q}$ of positive integers such that for every sequence $(H_q)_{q}$ of finite sets such that $|H_q|=n_q$ for all $q\in\nn$ and for every coloring $c:\bigcup_k\prod_{q=0}^{k-1}H_q\rightarrow \{0,1\}$ there exist $J_q\subseteq H_q$ such that $|J_q|=m_q$ for all $q\in\nn$ and such that $c$ is constant on $\prod_{q=0}^{k-1}J_q$ for infinitely many $k.$ More precisely, in this paper we prove the following density result, where by $\nn$ we denote the set of the natural numbers starting by 0 and by $\nn_+$ we denote the set of the positive natural numbers.

\begin{thm}
  \label{main_result}
  Let $\delta$ be a real number with $0<\delta\mik1$. Then there exists a map $V_{\delta}:\nn_+^{<\nn}\times\nn_+^{<\nn}\to\nn$ with the following property. For every sequence $(m_q)_{q}$ of positive integers, infinite sequence $(H_q)_q$ of finite sets,  $L$ infinite subset of $\nn_+$ and sequence $(D_k)_{k\in L}$ such that
  \begin{enumerate}
    \item[(a)] $|H_0|\meg V_\delta\big((m_0),\emptyset\big)$ and $|H_q|\meg V_\delta\big((m_p)_{p=0}^q,(|H_p|)_{p=0}^{q-1}\big)$ for all $q\in\nn_+$, and
    \item[(b)] $D_k$ is a subset of $\prod_{q=0}^{k-1}H_q$ of density at least $\delta$, for all $k\in L$,
  \end{enumerate}
  there exist a sequence of finite sets $(I_q)_q$ and an infinite subset $L'$ of $L$ such that
  \begin{enumerate}
    \item[(i)] $I_q\subseteq H_q$ and $|I_q|=m_q$ for all $q\in\nn$ and
    \item[(ii)] $\prod_{q=0}^{k-1}I_q\subseteq D_k$ for all $k\in L'$.
  \end{enumerate}
\end{thm}

For the definition of the map $V_\delta$, we will need some auxiliary maps. In the next section we define these maps and we describe their properties.

\section{The definition of the map $V_\delta$.}

For every $0<\theta<\ee\mik 1$ and every integer $k\meg 2$ we set
\[
\Sigma(\theta,\ee,k)=\Big\lceil \frac{k(k-1)}{2(\ee^k-\theta^k)}\Big\rceil.
\]
We will need the following elementary fact. Although it is well known, we could not find a reference and we include its proof
for the convenience of the reader.
\begin{lem} \label{correlation}
Let $0< \theta< \ee\mik 1$ and $k,N\in\nn$ with $k\meg 2$ and $N\meg\Sigma(\theta,\ee,k)$. Also let $(A_i)_{i=1}^N$ be a family of measurable
events in a probability space $(\Omega,\Sigma,\mu)$ such that $\mu(A_i)\meg\ee$ for all $1\mik i\mik N$. Then there exists a subset $F$ of $\{1,\ldots,N\}$
of cardinality $k$ such that
\[
\mu\Big( \bigcap_{i\in F} A_i\Big) \meg \theta^k.
\]
\end{lem}
\begin{proof}
Let $\mathcal{A}$ be the set of all functions $\sigma:[k]\to [N]$ and $\mathcal{B}=\{\sigma\in\mathcal{A}:\sigma \text{ is 1-1}\}$. Notice that
\begin{equation} \label{2pr-e3}
|\mathcal{A}\setminus\mathcal{B}|\mik \frac{k(k-1)}{2} N^{k-1}.
\end{equation}
By our assumptions and Jensen's inequality, we see that
\begin{equation}
  \begin{split}
  \label{2pr-e4}
    \ee^kN^k\mik &\Big( \int \sum_{i=1}^{N} \mathbf{1}_{A_i} d\mu \Big)^k \mik\int\Big(\sum_{i=1}^{N} \mathbf{1}_{A_i}\Big)^k d\mu
    =  \int \sum_{\sigma\in\mathcal{A}} \prod_{i=1}^{k} \mathbf{1}_{A_{\sigma(i)}} d\mu\\
= & \sum_{\sigma\in\mathcal{A}} \mu\Big( \bigcap_{i=1}^{k} A_{\sigma(i)}\Big)
  = \sum_{\sigma\in\mathcal{A}\setminus \mathcal{B}} \mu\Big( \bigcap_{i=1}^{k} A_{\sigma(i)}\Big) +
\sum_{\sigma\in\mathcal{B}} \mu\Big( \bigcap_{i=1}^{k} A_{\sigma(i)}\Big) \\
 \stackrel{\eqref{2pr-e3}}{\mik} & \frac{k(k-1)}{2} N^{k-1} +
\sum_{\sigma\in\mathcal{B}} \mu\Big( \bigcap_{i=1}^{k} A_{\sigma(i)}\Big).
  \end{split}
\end{equation}
Since $N\meg \Sigma(\theta,\ee,k)$ we get
\[
\frac{1}{|\mathcal{B}|} \sum_{\sigma\in\mathcal{B}}\ \mu \Big(\bigcap_{i=1}^k A_{\sigma(i)}\Big) \meg
\frac{1}{N^k} \sum_{\sigma\in\mathcal{B}} \mu \Big(\bigcap_{i=1}^k A_{\sigma(i)}\Big)
\stackrel{\eqref{2pr-e4}}{\meg} \ee^k- \frac{1}{N}\cdot \frac{k(k-1)}{2}\meg \theta^k
\]
and the proof is completed.
\end{proof}

For every $0<\ee\mik1$ we define the map $T_\ee:\nn_+^{<\nn}\to\nn$ as follows. For every $(m_q)_{q=0}^k$ in $\nn_+^{<\nn}$ we set
\begin{equation}
  \label{e02}
  T_\ee\big((m_q)_{q=0}^k\big)=\frac{2}{\ee'}\Sigma(\ee'/4,\ee'/2,m_k),
\end{equation}
where $\ee'=\ee$ if $k=0$ and $\ee'=\ee^{\prod_{q=0}^{k-1}m_q}2^{-2\sum_{j=0}^{k-1}\prod_{q=j}^{k-1}m_q}$ otherwise. The crucial property characterizing the map $T_\ee$ is given by the following lemma. For its proof we will need the following notation.

\begin{notation}
  Let $k_0<k_1< k_2$ be nonnegative integers and $(H_q)_q$ be a sequence of nonempty finite sets. Also let $x\in\prod_{q=k_0}^{k_1-1}H_q$ and $y\in\prod_{q=k_1}^{k_2-1}H_q$. By $x^\con y$ we denote the concatenation of the sequences $x,y$, i.e. the sequence $z\in\prod_{q=k_0}^{k_2-1}H(q)$ satisfying $z(q)=x(q)$ for all $q=k_0,\ldots,k_1-1$ and $z(q)=y(q)$ for all $q=k_1,\ldots,k_2-1$. Moreover, for
  $A\subseteq\prod_{q=k_0}^{k_1-1}H_q$ and $B\subseteq\prod_{q=k_1}^{k_2-1}H_q$ we set
  \[x^\con B=\{x^\con y:y\in B\}\]
  and
  \[A^\con B=\bigcup_{x\in A}x^\con B.\]
\end{notation}

\begin{lem}
  \label{T1_property_prequel}
  Let $k\in\nn$, $\ee$ be a real with $0<\ee\mik1$ and $m_0,\ldots,m_k$ be positive integers. Also let $(H_q)_{q=0}^k$ be a finite sequence of nonempty finite sets, satisfying $|H_q|\meg T_\ee\big((m_p)_{p=0}^q\big)$ for all $q=0,\ldots,k$ and a subset $D$ of $\prod_{q=0}^k H_q$ of density at least $\ee$. Then there exists a sequence $(I_q)_{q=0}^k$ such that
  \begin{enumerate}
    \item[(i)] $I_q\subseteq H_q$ and $|I_q|=m_q$ for all $q=0,\ldots,k$ and
    \item[(ii)] $\prod_{q=0}^k I_q\subseteq D$.
  \end{enumerate}
\end{lem}

\begin{proof}
  If $k=0$ then the result is immediate. Indeed, in this case we have that
  \[|D|\meg\ee\cdot|H_0|\meg\ee\cdot T_{\ee}\big((m_0)\big)\meg m_0.\]
  Suppose that $k\meg1$. Let $\ee_0=\ee$ and $\ee_{q+1}=(\ee_q/4)^{m_q}$, for all $q=0,\ldots,k-1$. It is easy to check that
  \[\ee_{q}=\ee^{\prod_{p=0}^{q-1}m_p}2^{-2\sum_{j=0}^{q-1}\prod_{p=j}^{q-1}m_p},\;\text{for all}\; q=1,\ldots,k.\]
  We inductively construct sequences $(I_q)_{q=0}^{k-1}$ and $(D_q)_{q=0}^{k}$ such that for every $q=0,\ldots,k-1$ we have
  \begin{enumerate}
    \item[(a)] $I_q\subset H_q$ and $|I_q|=m_q$,
    \item[(b)] $D_q\subset\prod_{p=q}^kH_q$ of density at least $\ee_q$ and
    \item[(c)] $I_q^\con D_{q+1}\subseteq D_q$.
  \end{enumerate}
  We set $D_0=D$ and assume that for some $q\in\{0,\ldots,k-1\}$, $(D_p)_{p=0}^q$ and if $q\meg1$ also $(I_p)_{p=0}^{q-1}$ have been constructed satisfying (a),(b) and (c) above. For every $x\in H_q$, we set
  \[D'_x=\{y\in\prod_{p=q+1}^k H_p:(x)^\con y\in D_q\}.\]
  Clearly the average of the densities of $D'_x$ as $x$ runs in $H_q$ is equal to the density of $D_q$ and therefore, by the inductive assumption (b) above, at least $\ee_q$. By a Fubini type argument, we have that the set
  \[B=\{x\in H_q:\;D'_x\;\text{is of density at least}\;\ee_q/2\}\]
  is of density at least $\ee_q/2$ in $H_q$. Thus
  \[\begin{split}
    |B|\meg(\ee_q/2)\;|H_q|\meg(\ee_q/2)\;T_\ee\big((m_p)_{p=0}^q\big)=\Sigma(\ee_q/4,\ee_q/2,m_q).
  \end{split}\]
  By Lemma \ref{correlation} there exists $I_q\subseteq B\subseteq H_q$ of cardinality $m_q$ such that $D_{q+1}=\cap_{x\in I_q}D'_x$ is of density at least $e_{q+1}$. The inductive step of the construction is complete.

  Using property (c) above it is easy to see that $I_1^\con\ldots^\con I_{k-1}\;^\con D_k\subseteq D$, where $D_k$ is subset of $H_k$ of density at least $\ee_k$. Thus
  \[|D_k|\meg\ee_k\;T_{\ee}\big((m_p)_{p=0}^k\big)\meg m_k\]
  and therefore we may pick $I_k\subset D_k\subset H_k$ of cardinality $m_k$. Clearly $(I_q)_{q=0}^k$ is the desired one and the proof is complete.
\end{proof}

\begin{defn}
  Let $(H_q)_q$ be a sequence of nonempty finite sets, $L$ be an infinite subset of $\nn_+$, $k_0\in\nn$ and $0<\ee\mik 1$. We will say that the sequence $(D_k)_{k\in L}$ is $(k_0,\ee)$-dense in $(H_q)_q$ if for every $k> k_0$ in $L$ we have that $D_k$ is a subset of $\prod_{q=k_0}^{k-1}H_q$ of density at least $\ee$, i.e.
  \[\frac{|D_k|}{|\prod_{q=k_0}^{k-1}H_q|}\meg\ee.\]
  Finally, we will say that $(D_k)_{k\in L}$ is $\ee$-dense in $(H_q)_q$, if it is $(0,\ee)$-dense in $(H_q)_q$.
\end{defn}
The next result is an immediate consequence of Lemma \ref{T1_property_prequel}.
\begin{cor}
  \label{T1_property}
  Let $(m_q)_{q}$ be a sequence of positive integers, $\ee$ be a real number with $0<\ee\mik1$, $k_0$ be a nonnegative integer and $(H_q)_q$ be a sequence of nonempty finite sets satisfying $|H_k|\meg T_\ee\big((m_q)_{q=k_0}^k\big)$ for all integers $k$ with $k\meg k_0$. Also let $L$ be an infinite subset of $\nn_+$ and $(D_k)_{k\in L}$ be $(k_0,\ee)$-dense in $(H_q)_q$. Then for every $k\in L$ with $k> k_0$ there exists a finite sequence $(I_q^k)_{q=k_0}^{k-1}$ of finite sets satisfying
  \begin{enumerate}
    \item[(i)] $I_q^k\subseteq H_q$ and $|I_q^k|=m_q$, for all $q=0,\ldots,k-1$ and
    \item[(ii)] $\prod_{q=0}^{k-1} I_q^k\subseteq D_k$.
  \end{enumerate}
\end{cor}

For every reals $\ee,\theta$ with $0<\theta<\ee\mik1$ and every positive integer $r$, we define the map $Q^r_{\theta,\ee}:\nn_+^{<\nn}\to\nn$ as follows. For every $(m_q)_{q=0}^k\in \nn_+^{<\nn}$, we set
\[
%\begin{equation}
%  \label{e03}
  Q^r_{\theta,\ee}\big((m_q)_{q=0}^k\big)=T_{\frac{1}{8}(\frac{\ee-\theta}{2^{r}})^2}\big((m_q)_{q=0}^k\big).
%\end{equation}
\]
By the definition of the the map $Q^r_{\theta,\ee}$, Corollary \ref{T1_property} has the following result as an immediate consequence.
\begin{cor}
  \label{Q_property_on_cut_cor1}
  Let $(m_q)_{q}$ be a sequence of positive integers, $\theta,\ee$ be real numbers with $0<\theta<\ee\mik1$, $k_0,r\in\nn$ with $r\meg1$ and $(H_q)_q$ be a sequence of nonempty finite sets such that
  $|H_k|\meg Q^r_{\theta,\ee}\big((m_q)_{q=k_0}^k\big)$, for all $k\meg k_0$.
  Also let $L$ be an infinite subset of $\nn_+$ and $(D_k)_{k\in L}$ be $(k_0,\frac{1}{8}(\frac{\ee-\theta}{2^{r}})^2)$-dense in  $(H_q)_q$. Then for every $k\in L$ with $k>k_0$ there exists a finite sequence $(I_q^k)_{q=k_0}^{k-1}$ of finite sets satisfying
  \begin{enumerate}
    \item[(i)] $I_q^k\subseteq H_q$ and $|I_q^k|=m_q$ for all $q=k_0,\ldots,k-1$ and
    \item[(ii)] $\prod_{q=k_0}^{k-1} I_q^k\subseteq D_k$.
  \end{enumerate}
\end{cor}

\begin{lem}
  \label{fubini_arg}
  Let $i,r$ be positive integers, $k_0\in\nn$ with $k_0<i$ and $\theta,\ee$ be reals with $0<\theta<\ee\mik1$. Also let $(H_q)_q$ be a sequence of nonempty finite sets satisfying $|\prod_{q=k_0}^{i-1}H_q|\mik r$, $L$ be an infinite subset of $\nn_+$ and $(D_k)_{k\in L}$ be $(k_0,\ee)$-dense in $(H_q)_q$. Then there exist $\Gamma\subset\prod_{q=k_0}^{i-1}H_q$ of density at least $\theta$, $L'$ infinite subset of $L$ and $(\widetilde{D}_k)_{k\in L'}$ being $(i,(\ee-\theta)/2^r)$-dense in $(H_q)_q$ such that $\Gamma^\con \widetilde{D}_k\subseteq D_k$ for all $k\in L'$.
\end{lem}

\begin{proof}
  By passing to a final segment of $L$, if it is necessary, we may assume that $\min L> i$. For every $k\in L$ we have the following. For every $y\in\prod_{q=i}^{k-1}H_q$ we set $\Gamma_y=\{x\in\prod_{q=k_0}^{i-1}H_q:x^\con y\in D_k\}$. Observe that the average of the densities of $\Gamma_y$ as $y$ runs in $\prod_{q=i}^{k-1}H_q$ is equal to the density of $D_k$ and therefore at least $\ee$. By a Fubini type argument we have that the set
  \[D_k'=\Big\{y\in\prod_{q=i}^{k-1}H_q:\Gamma_y\;\text{is of density at least}\;\theta\;\text{in}\;\prod_{q=k_0}^{i-1}H_q\Big\}\]
  is of density at least $\ee-\theta$ in $\prod_{q=i}^{k-1}H_q$. Since the powerset of $\prod_{q=k_0}^{i-1}H_q$ is of cardinality at most $2^r$, we have that there exist $\Gamma_k\subseteq \prod_{q=k_0}^{i-1}H_q$ of density at least $\theta$ and a subset $\widetilde{D}_k$ of $D'_k$ of density at least $(\ee-\theta)/2^r$ in $\prod_{q=i}^{k-1}H_q$ such that $\Gamma_y=\Gamma_k$ for all $y\in \widetilde{D}_k$.

  Passing to an infinite subset $L'$ of $L$, we stabilize $\Gamma_k$ into some $\Gamma$, which of course is of density at least $\theta$ and the proof is complete.
\end{proof}

\begin{cor}
  \label{Q_property_lem2}
  Let $(m_q)_{q}$ be a sequence of positive integers, $i,r$ be positive integers, $k_0\in\nn$ with $k_0<i$ and $\theta,\ee$ be reals with $0<\theta<\ee\mik1$. Also let $(H_q)_q$ be a sequence of nonempty finite sets such that
  \begin{enumerate}
    \item[(a)] $|\prod_{q=k_0}^{i-1}H_q|\mik r$ and
    \item[(b)] $|H_k|\meg Q^r_{\theta,\ee}\big((m_q)_{q=i}^k\big)$, for all $k\meg i$.
  \end{enumerate}
  Finally, let $L$ be an infinite subset of $\nn_+$ and $(D_k)_{k\in L}$ be $(k_0,\ee)$-dense in $(H_q)_q$. Then there exist a subset $\Gamma$ of $\prod_{q=k_0}^{i-1}H_q$ of density at least $\theta$ and an infinite subset $L'$ of $L$ such that for every $k\in L'$ with $k>i$ there exists a finite sequence $(I_q^k)_{q=i}^{k-1}$ of finite sets satisfying
  \begin{enumerate}
    \item[(i)] $I_q^k\subseteq H_q$ and $|I_q^k|=m_q$, for all $q=i,\ldots,k-1$ and
    \item[(ii)] $\Gamma^\con\prod_{q=i}^{k-1} I_q^k\subseteq D_k$.
  \end{enumerate}
\end{cor}
\begin{proof}
  By Lemma \ref{fubini_arg}, there exist $\Gamma\subset\prod_{q=k_0}^{i-1}H_q$ of density at least $\theta$, $L'$ infinite subset of $L$ and $(D'_k)_{k\in L'}$ being $(i,(\ee-\theta)/2^r)$-dense in $(H_q)_q$ such that
  \begin{equation}
    \label{e13}
    \Gamma^\con D'_k\subseteq D_k
  \end{equation}
  for all $k\in L'$.
  By Corollary \ref{Q_property_on_cut_cor1}, we have that for every $k\in L'$
  there exists a finite sequence $(I_q^k)_{q=i}^{k-1}$ of finite sets satisfying
  \begin{enumerate}
    \item[(i)] $I_q^k\subseteq H_q$ and $|I_q^k|=m_q$ for all $q=k_0,\ldots,k-1$ and
    \item[(ii)] $\prod_{q=k_0}^{k-1} I_q^k\subseteq D_k'$.
  \end{enumerate}
  By (ii) above and \eqref{e13}, we get that $\Gamma^\con\prod_{q=k_0}^{k-1}I_q^k\subseteq D_k$ for all $k\in L'$.
\end{proof}

Finally, in order to define the map $V_\delta$, we need some additional notation. For every $s\in\nn$ we set \[\Delta_s=\{i/2^s: 0<i\mik2^s\;\text{in}\;\nn\},\]
$\Delta=\cup_{s\in\nn}\Delta_s$
and for every real $\delta$ with $0<\delta\mik 1$, we pick $s_\delta\in\nn$ such that $2^{1-s_\delta}\mik\delta$. Let $\delta$ be a real with $0<\delta\mik1$. We define the map $V_\delta:\nn_+^{<\nn}\times\nn_+^{<\nn}\to\nn$ as follows. For every $m_0\in\nn_+$ we set
\[V_\delta((m_0),\emptyset)=\max_{\ee\in\Delta_{s_\delta}}T_\ee\big((m_0)\big)\]
where $\emptyset$ denotes the empty sequence and for every $(m_q)_{q=0}^k,(n_q)_{q=0}^{k-1}\in\nn_+^{<\nn}$ with $k>0$ we set
\[V_\delta\big((m_q)_{q=0}^k,(n_q)_{q=0}^{k-1}\big)=\max\Big\{T_{\min \Delta_{s_\delta}}\big((m_q)_{q=0}^k\big),\max_{\substack{i=1,\ldots,k\\\theta<\ee\;\text{in}\;\Delta_{s_\delta+k}}} Q^{r_i}_{\theta,\ee}\big((m_q)_{q=i}^k\big)\Big\},\]
where $r_i=\prod_{q=0}^{i-1}n_q$. We arbitrarily extend the map $V_\delta$ to the remaining elements of $\nn_+^{<\nn}\times\nn_+^{<\nn}$.

\section{Proof of Theorem \ref{main_result}}
For the rest of this section we fix a real $\delta$ with $0<\delta\mik1$, sequences $(m_q)_{q},(n_q)_{q}$ of positive integers satisfying $n_k\meg V_\delta\big((m_q)_{q=0}^k,(n_q)_{q=0}^{k-1}\big)$ for all $k\in\nn$ and a sequence $(H_q)_q$ of finite sets with $|H_q|=n_q$ for all $q\in\nn$.
\subsection{The inductive hypothesis}
The main part of the proof consists of showing that for every countable ordinal $\xi$ the property $\mathcal{P}(\xi)$ (see Definition \ref{Defn_Property_P} below) holds.
In order to state the definition of the property $\mathcal{P}(\xi)$ we need some additional notation. For $L$ an infinite subset of $\nn$, by $[L]^{<\infty}$ we denote the set of all finite subsets of $L$.
\begin{notation}
  Let $i$ be a positive integer, $k_0\in\nn$ with $k_0<i$, $\Gamma$ be a subset of $\prod_{q=k_0}^{i-1}H_q$, $L$ an infinite subset of $\nn_+$ and $(D_k)_{k\in L}$ such that $D_k\subseteq\prod_{q=k_0}^{k-1}H_q$ for all $k\in L$ with $k>k_0$. We set
  \[\begin{split}
    \mathcal{F}(i,\Gamma,(D_k)_{k\in L})=\{F\in&[L]^{<\infty}:F\neq\emptyset,\;\min F>i\;\text{and there exists a sequence}\\
    &(I_q)_{q=i}^{\max F-1}
  \text{such that}\;I_q\subseteq H_q\;\text{and}\;|I_q|=m_q\;\text{for all}\\
  &q=i,\ldots,\max F-1\;\text{and}\;\Gamma^\con\prod_{q=i}^{k-1}I_q\subseteq D_k\;\text{for all}\;k\in F\}
  \end{split}\]
\end{notation}

Let $i,\Gamma$ and $(D_k)_{k\in L}$ as above. It is easy to see that the family $\mathcal{F}(i,\Gamma,(D_k)_{k\in L})$ satisfies the following property. For every nonempty sets $A\subseteq B$ such that $B\in\mathcal{F}(i,\Gamma,(D_k)_{k\in L})$, we have $A\in\mathcal{F}(i,\Gamma,(D_k)_{k\in L})$. That is $\mathcal{F}(i,\Gamma,(D_k)_{k\in L})\cup\{\emptyset\}$ is hereditary. For a hereditary family $\mathcal{F}$, the rank of $\mathcal{F}$ is defined as follows. If $\mathcal{F}$ is not compact, we define the rank of $\mathcal{F}$ to be equal to $\omega_1$. Otherwise, for every $\sqsubseteq$-maximal element $F\in\mathcal{F}$ we set $r_{\mathcal{F}}(F)=0$ and recursively for every $F\in\mathcal{F}$ we set $r_{\mathcal{F}}(F)=\sup\{r_{\mathcal{F}}(G)+1:G\in\mathcal{F}\;\text{and}\;F\sqsubset G\}$. The rank of $\mathcal{F}$ is defined to be equal to $r_{\mathcal{F}}(\emptyset)$.
The rank of a family $\mathcal{F}(i,\Gamma,(D_k)_{k\in L})$ is defined to be equal to the rank of the hereditary family $\mathcal{F}(i,\Gamma,(D_k)_{k\in L})\cup\{\emptyset\}$. Finally, for $M$ infinite subset of $\nn$ and $\xi$ countable ordinal, we will say that $\mathcal{F}(i,\Gamma,(D_k)_{k\in L})$ is of hereditary rank at least $\xi$ in $M$, if for every $M'$ infinite subset of $M$, the hereditary family $[M']^{<\infty}\cap(\mathcal{F}(i,\Gamma,(D_k)_{k\in L})\cup\{\emptyset\})$ is of rank at least $\xi$.

\begin{defn}
  \label{Defn_Property_P}
  Let $\xi$ be a countable ordinal number. We will say that $\mathcal{P}(\xi)$ holds, if for every $\theta<\ee$ in $\Delta$, every $i'>k_0$ in $\nn$, every infinite subset $L$ of $\nn_+$ and $(D_k)_{k\in L}$ being $(k_0,\ee)$-dense in $(H_q)_q$, there exist  an integer $i$ with $i\meg i'$, an infinite subset $L'$ of $L$ and some $\Gamma\subseteq\prod_{q=k_0}^{i-1} H_q$ of density at least $\theta$ such that the family $\mathcal{F}(i,\Gamma,(D_k)_{k\in L})$
  is of hereditary rank at least $\xi$ in $L'$.
\end{defn}

\begin{cor}
  \label{P(1)_holds}
  The property $\mathcal{P}(1)$ holds.
\end{cor}
\begin{proof}
  Let $\theta<\ee$ in $\Delta$, $i'>k_0$ in $\nn$, an infinite subset $L$ of $\nn_+$ and $(D_k)_{k\in L}$ being $(k_0,\ee)$-dense in $(H_q)_q$. We pick $i\meg i'$ such that $\theta,\ee\in\Delta_{s_\delta+i}$. The result is immediate by Corollary \ref{Q_property_lem2}.
\end{proof}

\begin{notation}
  Let $0<k_0< c\mik k$ be integers and $x\in\prod_{q=k_0}^{k-1}H_q$. By $x\upharpoonright c$, we denote the initial segment of $x$ in $\prod_{q=k_0}^{c-1}H_q$.
\end{notation}

\begin{lem}
  \label{induction_main_lemma}
  Assume that for some sequence of countable ordinals $(\xi_n)_{n}$ we have that $\mathcal{P}(\xi_n)$ holds for all $n\in\nn$. Let $\eta$ in $\Delta$, $k_0< i'$ in $\nn$, $L$ be an infinite subset of $\nn_+$ and $(\widetilde{D}_k)_{k\in L}$ being $(k_0,\eta)$-dense in $(H_q)_q$. Then there exist an infinite subset $L'=\{l_0<l_1<\ldots\}$ of $L$ with $i'\mik l_0$, an infinite subset $I=\{i_0<i_1<\ldots\}$ of $\nn$ and $(\Gamma_{i})_{i\in I}$ being $(k_0,\eta^2/8)$-dense in $(H_q)_q$ such that and for every $n\in\nn$ we have
  \begin{enumerate}
    \item[(i)] $l_n<i_n<l_{n+1}$,
    \item[(ii)] for every $x\in\Gamma_{i_n}$ we have $x\upharpoonright l_n\in \widetilde{D}_{l_n}$
    \item[(iii)] the family $\mathcal{F}(i_n,\Gamma_{i_n},(\widetilde{D}_k)_{k\in L})$
  is of hereditary rank at least $\xi_n$ in $L'$.
  \end{enumerate}
\end{lem}
\begin{proof}
  Using $\mathcal{P}(\xi_n)$ we inductively construct a decreasing sequence $(L_n)_{n}$ of infinite subsets of $L$ and a $(k_0,\eta^2/8)$-dense sequence $(\Gamma_{i_n})_{n}$ in $(H_q)_q$ such that $i'\mik i_0$ and for every $n\in\nn$ we have that
  \begin{enumerate}
    \item[(a)] $\min L_n< i_n<\min L_{n+1}$,
    \item[(b)] for every $x\in\Gamma_{i_n}$ we have $x\upharpoonright\min L_n\in \widetilde{D}_{\min L_n}$ and
    \item[(c)] the family $\mathcal{F}(i_n,\Gamma_{i_n},(\widetilde{D}_k)_{k\in L})$
  is of hereditary rank at least $\xi_n$ in $L_{n+1}$.
  \end{enumerate}
  First by Ramsey's Theorem and Lemma \ref{correlation} we may pass to an infinite subset $L_0$ of $L$ such that $\min L_0> i'$ and for every $k<k'$ in $L_0$ we have that the set $\{x\in \widetilde{D}_{k'}:x\upharpoonright k\in\widetilde{D}_k\}$ is of density at least $\eta^2/4$ in $\prod_{q=k_0}^{k-1}H_q$.  Suppose that for some $n\in\nn$ we have constructed $(L_p)_{p=0}^n$ and if $n>0$, $(i_p)_{p=0}^{n-1}$ and $(\Gamma_{i_p})_{p=0}^{n-1}$. Treating $i'$ as $i_{-1}$ in case $n=0$, the inductive step of the construction has as follows. For every $k\in L_n$ we set
  \begin{equation}
    \label{e10}
    \widetilde{D}_k^n=\{x\in \widetilde{D}_k:x\upharpoonright\min L_n\in \widetilde{D}_{\min L_n}\}.
  \end{equation}
  Since $L_n$ is subset of $L_0$, by the choice of $L_0$, we have that $(\widetilde{D}^n_k)_{k\in L'_n}$ is $(k_0,\eta^2/4)$-dense in $(H_q)_q$.
  Applying $\widetilde{\mathcal{P}}(\xi_n)$, we obtain some $i_n>\min L_n$, an infinite subset $L_{n+1}$ of $L_n$ and $\Gamma_{i_n}\subseteq\prod_{q=k_0}^{i_n-1}H_q$ of density at least $\eta^2/8$ such that the family $\mathcal{F}(i_n,\Gamma_{i_n},(\widetilde{D}^n_k)_{k\in L})$ is of hereditary rank at least $\xi_n$ in $L_{n+1}$. Passing to a final segment of $L_{n+1}$ we may assume that $\min L_{n+1}>i_n$ and the inductive step of the construction is complete. Indeed, property (a) is immediate, as well as property (c) by the observation that $\mathcal{F}(i_n,\Gamma_{i_n},(\widetilde{D}^n_k)_{k\in L})\subseteq \mathcal{F}(i_n,\Gamma_{i_n},(\widetilde{D}_k)_{k\in L})$. Concerning (b), notice that, since $\mathcal{F}(i_n,\Gamma_{i_n},(\widetilde{D}^n_k)_{k\in L})$ in non empty there exist $k\in L_{n+1}$ and $y$ such that $x^\con y\in \widetilde{D}_k^n$ for all $x\in\Gamma_{i_n}$. Since $\min L_n<i_n$, by \eqref{e10}, we have that $x\upharpoonright\min L_n=x^\con y\upharpoonright\min L_n\in\widetilde{D}_{\min L_n}$.

  Setting $L'=\{\min L_n:n\in\nn\}$,  it is easy to check that the proof is complete.
\end{proof}

\begin{lem}
  \label{induction_consequence_lemma}
  Assume that for some (not necessarily strictly) increasing sequence of countable ordinals $(\xi_n)_{n}$ we have that $\mathcal{P}(\xi_n)$ holds for all $n\in\nn$. Then $\mathcal{P}(\xi)$ holds, where $\xi=\sup\{\xi_n+1:n\in\nn\}$.
\end{lem}
\begin{proof}
  Let $\theta<\ee$ in $\Delta$, $i'>k_0$ in $\nn$, an infinite subset $L$ of $\nn_+$ and $(D_k)_{k\in L}$ being $(k_0,\ee)$-dense in $(H_q)_q$. We pick $i\meg i'$ such that $\theta,\ee\in\Delta_{s_\delta+i}$. By Lemma \ref{fubini_arg} there exist $\Gamma\subset\prod_{q=k_0}^{i-1}H_q$ of density at least $\theta$, $L''$ infinite subset of $L$ and $(\widetilde{D}_k)_{k\in L''}$ being $(i,(\ee-\theta)/2^r)$-dense in $(H_q)_q$ such that
  \begin{equation}
    \label{e12}
    \Gamma^\con \widetilde{D}_k\subseteq D_k
  \end{equation}
  for all $k\in L''$. By Lemma \ref{induction_main_lemma} there exist an infinite subset $L'=\{l_0<l_1<\ldots\}$ of $L''$ with $i\mik l_0$, an infinite subset $I=\{i_0<i_1<\ldots\}$ of $\nn$ and $(\Gamma_{i})_{i\in I}$ being $(i,\frac{1}{8}(\frac{\ee-\theta}{2^r})^2)$-dense in $(H_q)_q$ such that and for every $n\in\nn$ we have that
  \begin{enumerate}
    \item[(i)] $l_n<i_n<l_{n+1}$,
    \item[(ii)] for every $x\in\Gamma_{i_n}$ we have $x\upharpoonright l_n\in \widetilde{D}_{l_n}$
    \item[(iii)] the family $\mathcal{F}(i_n,\Gamma_{i_n},(\widetilde{D}_k)_{k\in L})$
  is of hereditary rank $\xi_n$ in $L'$.
  \end{enumerate}
  By the choice of $(n_q)_{q}$, the definition of the map $V_\delta$ and Corollary \ref{Q_property_on_cut_cor1}, we have that for every $n\in \nn$ there exists a finite sequence $(I_q^n)_{q=i}^{i_n-1}$ of finite sets satisfying
  \begin{enumerate}
    \item[(a)] $I_q^n\subseteq H_q$ and $|I_q^n|=m_q$ for all $q=i,\ldots,i_n-1$ and
    \item[(b)] $\prod_{q=i}^{i_n-1} I_q^n\subseteq \Gamma_{i_n}$.
  \end{enumerate}
  By (ii) and (b) we have that
  \begin{equation}
    \label{e11}
    \prod_{q=i}^{l_n-1}I_q^n\subseteq\widetilde{D}_{l_n}
  \end{equation}
   for all $n\in\nn$. \\
   \textbf{Claim:} $\mathcal{F}(i,\Gamma,(D_k)_{k\in L})\supseteq\{\{l_n\}\cup F:n\in\nn\;\text{and}\;F\in\mathcal{F}(i_n,\Gamma_{i_n},(\widetilde{D}_k)_{k\in L})\}$.
   \begin{proof}
     [Proof of Claim.]
     Let $n\in\nn$ and $F\in\mathcal{F}(i_n,\Gamma_{i_n},(\widetilde{D}_k)_{k\in L})$. Let $(I_q)_{q=i_n}^{\max F-1}$ the finite sequence witnessing that $F\in\mathcal{F}(i_n,\Gamma_{i_n},(\widetilde{D}_k)_{k\in L})$. Let $(J_q)_{q=i}^{\max F-1}$ defined by $J_q=I^n_q$ for all $q=i,\ldots,i_n-1$ and $J_q=I_q$ for all $q=i_n,\ldots, \max F-1$. By \eqref{e11}, (b) and the fact that $F\in\mathcal{F}(i_n,\Gamma_{i_n},(\widetilde{D}_k)_{k\in L})$ we have that
     \[\prod_{q=i}^{l_n-1}J_q\subset\widetilde{D}_{l_n}\;\text{and}\;\prod_{q=i}^{k-1}J_q\subseteq\widetilde{D}_k\;\text{for all}\;k\in F\]
     By \eqref{e12} we get that $\{l_n\}\cup F\in \mathcal{F}(i,\Gamma,(D_k)_{k\in L})$.
   \end{proof}
   Lemma \ref{induction_consequence_lemma} is an immediate consequence of the claim above.
\end{proof}
By Corollary \ref{P(1)_holds} and Lemma \ref{induction_consequence_lemma} we have the following.

\begin{cor}
  \label{P(every_countable_ordinal)_holds}
  For every countable ordinal $\xi$, the property $\mathcal{P}(\xi)$ holds.
\end{cor}

We are ready to complete the proof for Theorem \ref{main_result}.

\begin{proof}[Proof of Theorem \ref{main_result}]
  Let $L$ be an infinite subset of $\nn_+$ and $(D_k)_{k\in L}$ be $\delta$-dense in $(H_q)_q$.
 By Corollary \ref{P(every_countable_ordinal)_holds} we have that for every countable ordinal $\xi$ there exist a positive integer $i_\xi$ and a subset $\Gamma_{\xi}$ of $\prod_{q=0}^{i_\xi-1}H_q$ of density at least $\delta/2$ such that the family $\mathcal{F}(i_\xi,\Gamma_\xi,(D_k)_{k\in L})$ is of rank at least $\xi$.
 We pick $\mathcal{A}$ uncountable subset of $\omega_1$ such that $i_\xi$ and $\Gamma_\xi$ are stabilized into some $i$ and $\Gamma$ respectively, for all $\xi\in\mathcal{A}$. Let $\mathcal{F}=\mathcal{F}(i,\Gamma,(D_k)_{k\in L})$. Clearly the rank of $\mathcal{F}$ is greater that or equal to $\xi$, for all $\xi\in\mathcal{A}$. Since $\mathcal{A}$ is uncountable we get that the rank of $\mathcal{F}$ is $\omega_1$. The latter and the hereditariness of the family $\mathcal{F}$ yields the existence of some infinite subset $M=\{m_0<m_1<\ldots\}$ of $L$ such that for every $F$ finite subset of $M$ we have that $F\in \mathcal{F}$. Clearly, we may assume that $m_0>i$. Let us denote by $F_n$ the initial subset of $M$ of length $n+1$, i.e. $F_n=\{m_0,\ldots, m_n\}$, for every $n$ nonnegative integer. Thus for every positive integer $n$, since $F_n\in \mathcal{F}$, we have that there exist $(I_q^n)_{q=i}^{m_n-1}$ such that
  \begin{enumerate}
    \item[(1$'$)] $I_q^n\subseteq H_q$ and $|I_q^n|=m_q$ for all $q=i,\ldots,m_n-1$ and
    \item[(2$'$)] $\Gamma^\con\prod_{q=i}^{m_k-1}I^n_q\subset D_{m_k}$ for all $k=0,\ldots,n$.
  \end{enumerate}
  We pass to a subsequence $(F_{k_n})_{n}$ of $(F_n)_{n}$ such that $\big((I_q^{k_n})_{q=i}^{m_{k_n}-1}\big)_{n}$ pointwise converges to some $(I_q)_{q=i}^\infty$. It is immediate that the following are satisfied.
  \begin{enumerate}
    \item[(1)] For every integer $q\meg i$, $I_q\subset H_q$ and $|I_q|=m_q$.
    \item[(2)] For every $n\in\nn$, $\Gamma^\con\prod_{q=i}^{m_{k_n}-1}I_q\subset D_{m_{k_n}}$.
  \end{enumerate}
  Notice that $|H_q|\meg T_{\delta/2}\big((m_p)_{p=0}^q\big)$ for all $q=0,\ldots,i-1$. Since $\Gamma$ is subset of $\prod_{q=0}^{i-1}H_q$ of density at least $\delta/2$, by Corollary \ref{T1_property}, there exists $(I_q)_{q=0}^{i-1}$ such that $I_q\subset H_q$ and $|I_q|=m_q$ for all $q=0,\ldots,i-1$ and $\prod_{q=0}^{i-1}I_q\subset \Gamma$.
  It is easy to check that the sequence $(I_q)_{q}$ is the desired one.
\end{proof}

\section{Some calculus on the map $V_\delta$}
In this section we provide bounds for the map $V_\delta$ in terms of the Ackermann functions. In particular, we provide bounds for the map $f:\nn^{<\nn}\to\nn$ inductively defined as follows. We set $f\big((m_0)\big)=V_\delta\big((m_0),\emptyset\big)$ for all $m_0\in\nn_+$ and $f\big((m_q)_{q=0}^k\big)=V_\delta\big((m_q)_{q=0}^{k},\big(f\big((m_p)_{p=0}^q\big)\big)_{q=0}^{k-1}\big)$ for all $(m_q)_{q=0}^k\in\nn^{<\nn}$ with $k>0$.
 Let us recall that the Ackermann hierarchy of functions from $\nn$ into $\nn$ is defined as follows.
\[\begin{split}
  &A_0(0)=1,\;A_0(1)=2,\;\text{and}\;A_0(x)=2+x\;\text{for all}\;x\meg2\\
  &A_{n+1}(x)=A^{(x)}_n(1)\;\text{for all}\;x\in\nn,
\end{split}\]
where $A^{(x)}_n$ denotes the $x$th iteration of $A_n$. (The 0th iteration of any function is considered to be the identity.) In particular, according to the above definition we have that $A_1(x)=2x$ for all $x\meg1$ and $A_2(x)=2^x$ for all $x\in\nn$. All the functions $A_n$ are primitive recursive and more precisely $A_n$ belongs to the class $\mathcal{E}^{n+1}$ of the Grzegorczyk hierarchy for all $n\in\nn$.

For the rest of this section we fix a real $\delta$ and a sequence $(m_q)_{q}$ of integers satisfying the following.
\begin{enumerate}
  \item[(a)] $0<\delta\mik1/2$ and
  \item[(b)] $m_q\meg2$ for all $q\in\nn$.
\end{enumerate}
For every $0<\ee\mik\delta/2$, we set $p_\ee=\lceil\log(1/\ee)\rceil$. By (a) it is immediate that $p_\ee\meg2$ for all $0<\ee\mik\delta/2$.
Moreover, for every $k\in\nn$ we set $\overline{pm}_k=\prod_{q=0}^k m_q$ and $\overline{spm}_k=\sum_{p=0}^k\prod_{q=p}^km_q$.
\begin{lem}
  \label{bound_for_T_ee}
  For every nonnegative integer $k$ and every real $\ee$ with $0<\ee\mik\delta/2$, we have that $T_\ee\big((m_q)_{q=0}^k\big)\mik A_2(5p_\ee\overline{pm}_k)$.
\end{lem}
 \begin{proof}
   Fix $k$ and $\ee$ as above. For the quantity $\ee'$ participating in the definition of the map $T_\ee$ (see \eqref{e02}), we have
   \begin{equation}
     \label{e04} \ee'=\ee^{\overline{pm}_{k-1}}2^{-2\overline{spm}_{k-1}} \stackrel{(b)}{\meg}\ee^{\overline{pm}_{k-1}}2^{-2\overline{pm}_k}.
   \end{equation}
   Moreover
   \begin{equation}
     \label{e05}
     \begin{split}
       \Big(\frac{\ee'}{2}\Big)^{m_k}-\Big(\frac{\ee'}{4}\Big)^{m_k}&= \Big(\frac{\ee'}{2}\Big)^{m_k}(1-2^{-m_k}) \stackrel{(b),\eqref{e04}}{\meg}\ee^{\overline{pm}_k}2^{-2\overline{spm}_k+m_k-1}.
     \end{split}
   \end{equation}
   Thus
   \begin{equation}
     \label{e06}
     \begin{split}
     \Sigma\Big(\frac{\ee'}{4},\frac{\ee'}{2},m_k\Big) \stackrel{\eqref{e05}}{\mik}\lceil\ee^{-\overline{pm}_k}2^{2\overline{spm}_k}\rceil \mik A_2(p_\ee\overline{pm}_k+2\overline{spm}_k).
   \end{split}
   \end{equation}
   Finally
   \[\begin{split}
     T_\ee\big((m_q)_{q=0}^k\big)&= \frac{2}{\ee'}\Sigma\Big(\frac{\ee'}{4},\frac{\ee'}{2},m_k\Big)\\ &\stackrel{\eqref{e04},\eqref{e06}}{\mik}A_2(1+p_\ee\overline{pm}_{k-1}+2\overline{pm}_k)A_2(p_\ee\overline{pm}_k+2\overline{spm}_k)\\
     &\stackrel{(a),(b)}{\mik}A_2(5p_\ee\overline{pm}_k).
   \qedhere\end{split}\]
 \end{proof}

 \begin{cor}
   \label{bound_for_Q}
   For every $0<\ee\mik\delta/2$, $k\in\nn$ and positive integer $r$, we have $Q^r_{\theta,\ee}\big((m_q)_{q=0}^k\big)\mik A_2((15+10r+10p_{\ee-\theta})\overline{pm}_k)$.
 \end{cor}

 \begin{prop}
   \label{bound_for_S}
   For every integer $k\meg0$, we have $f\big((m_q)_{q=0}^k\big)\mik A_2^{(1+k)}(5s_\delta\overline{pm}_k)$.
 \end{prop}
 \begin{proof}
   We will prove it by induction on $k$. For $k=0$ it is immediate by the definition of the maps $V_\delta,f$ and Lemma \ref{bound_for_T_ee}. Assume that for some $k\meg1$ we have that lemma holds for $k-1$. We set $\alpha_j=A_2^{(1+j)}(5s_\delta\overline{pm}_j)$ for all $0\mik j<k$ and $\beta=\prod_{j=0}^{k-1}\alpha_j$. Then
   \[\beta\mik \prod_{j=0}^{k-1}A_2^{(j+1)}(5s_\delta\overline{pm}_j)=A_2\Big(\sum_{j=0}^{k-1}A_2^{(j)}(5s_\delta\overline{pm}_j)\Big).\]
   By the definition of the maps $V_\delta,f$, the inductive assumption, the above inequality and Corollary \ref{bound_for_Q} we have the following.
   \begin{equation}
     \begin{split}
     \label{e14}
     f\big((m_q)_{q=0}^k\big)& \mik Q^{\beta}_{\frac{1}{2^{s_\delta+k}},\frac{2}{2^{s_\delta+k}}}\big((m_q)_{q=0}^k\big)
     \mik A_2\big((15+10\beta+10(s_\delta+k))\overline{pm}_k\big)\\
     &= A_2\big((25+10s_\delta)\overline{pm}_{k}+10(k-1)\overline{pm}_{k}+10\beta\overline{pm}_{k}\big).
     \end{split}
   \end{equation}
   In order to bound $f\big((m_q)_{q=0}^k\big)$ we will need the following elementary inequalities.
   \begin{enumerate}
     \item[($1'$)] $20x-10x\log x\mik 2^x$ for all integers $x\meg8$ and
     \item[($2'$)] $(25+10x)y\mik 2^{xy}$ for all integers $x,y$ with $x\meg2$ and $y\meg4$.
   \end{enumerate}
   Observing that $\overline{pm}_k\meg 2^{k+1}$, the above inequalities in particular yield the following.
   \begin{enumerate}
     \item[(1)] $10(k-1)\overline{pm}_k\mik A_2(\overline{pm}_k)$ for all $k>1$ and
     \item[(2)] $(25+10s_\delta)\overline{pm}_k\mik A_2(s_\delta\overline{pm}_k)$ for all $k\meg1$.
   \end{enumerate}
   By the latter inequalities and distinguishing cases for $k=1$ and $k>1$ one can easily derive the following.
   \begin{equation}
     \label{e15}
     10(k-1)\overline{pm}_k+(25+10s_\delta)\overline{pm}_k\mik A_2^{(k)}(s_\delta\overline{pm}_k).
   \end{equation}
   On the other hand, using the elementary inequality $10x\mik 2^{2x}$ for all integers $x\meg4$, the fact that $s_\delta\meg2$ and $5s_\delta\overline{pm}_{k-1}+1\mik3s_\delta\overline{pm}_{k}$, we get that
   \begin{equation}
     \label{e16}
     10\beta\overline{pm}_{k}\mik A_2^{(k)}(4s_\delta\overline{pm}_{k})
   \end{equation}
   The result follows easily by applying \eqref{e15} and \eqref{e16} to \eqref{e14}.
 \end{proof}

%-------------------------------------------------------------------%
%                           Bibliography                            %
%-------------------------------------------------------------------%

\end{document}